\documentclass[12 pt,reqno]{amsart}
\usepackage{amscd,amssymb,amsthm}
\usepackage[arrow,matrix]{xy}
\usepackage{color, colortbl}
\usepackage{graphicx}
\usepackage{cite}
\oddsidemargin=0.3in \evensidemargin=0.3in \theoremstyle{plain}
\newtheorem{thm}[subsection]{Theorem}
\newtheorem{lem}[subsection]{Lemma}

\newtheorem{cor}[subsection]{Corollary}

\newtheorem{rem}[subsection]{Remark}
\theoremstyle{definition}
\numberwithin{equation}{section} 

\newcommand{\BID}{{\rm BID}}

\newcommand{\beq}{\begin{eqnarray}}
\newcommand{\eeq}{\end{eqnarray}}

\newcommand{\beqs}{\begin{eqnarray*}}
\newcommand{\eeqs}{\end{eqnarray*}}

\begin{document}
%\title [\texttt{A Note on the Extremal $ \MakeLowercase{(n,m)}$-Graphs with respect to Bond Incident Degree Indices}]{A Note on the Extremal $(n,m)$-Graphs with respect to Bond Incident Degree Indices}
\title [\texttt{On the Extremal Graphs with respect to Bond Incident Degree Indices}]{On the Extremal Graphs with respect to Bond Incident Degree Indices}
%\title [\texttt{A Note on the Extremal Graphs with respect to Bond Incident Degree Indices}]{A Note on the Extremal Graphs with respect to Bond Incident Degree Indices}
\author{Akbar Ali$^ \MakeLowercase{a}$ and Darko Dimitrov$^  \MakeLowercase{b}$}
 \address{$^{a}$Department of Mathematics, University of Management \& Technology, Sialkot-Pakistan.}
\email{akbarali.maths@gmail.com}
 \address{$^{b}$Hochschule f\"{u}r Technik und Wirtschaft Berlin, Germany \& Faculty of Information Studies, Novo mesto, Slovenia.}
\email{darko.dimitrov11@gmail.com}

%\subjclass{??????????????}%

\begin{abstract}
Many existing degree based topological indices can be clasified as bond incident degree (BID) indices,
whose general form is $\BID(G)=$ $\sum_{uv\in E(G)}$ $\Psi(d_{u},d_{v})$, where $uv$ is the edge
connecting the vertices $u,v$ of the graph $G$, $E(G)$ is the edge set of $G$, $d_{u}$ is the
degree of the vertex $u$ and $\Psi$ is a non-negative real valued (symmetric) function of $d_{u}$
and $d_{v}$. Here, it has been proven that if the extension of $\Psi$ to the interval $[0,\infty)$
satisfies certain conditions then the extremal $(n,m)$-graph with respect to the BID index
(corresponding to $\Psi$) must contain at least one vertex of degree $n-1$.
It has been shown that these conditions are satisfied for the general sum-connectivity index (whose special cases are:
the first Zagreb index and the Hyper Zagreb index), for the general Platt index (whose special cases are:
the first reformulated Zagreb index and the Platt index) and for the variable sum exdeg index.
Applying aforementioned result, graphs with maximum values of the aforementioned BID indices
among tree, unicyclic, bicyclic, tricyclic and tetracyclic graphs were characterized.
Some of these results are new and the already existing results are proven in a
shorter and more unified way.

%Focus is made
%on characterizing the tree, unicyclic, bicyclic, tricyclic and tetracyclic graphs having maximum
%aforementioned BID indices. Consequently, several already reported results are obtained as corollaries.

\end{abstract}
\maketitle

\section{Introduction}

It is well known fact that molecules can be represented by graphs in which vertices correspond to the atoms while edges represent the covalent bonds between atoms \cite{EsBo2013,Tr92}. Predicting physicochemical properties of molecules is considered to be a fascinating issue in theoretical chemistry. Many predictive methods have been development for correlating molecular structures with their physicochemical properties. One of the simplest such methods involves topological indices \cite{Balaban13}. Topological indices are numerical quantities of a graph which are invariant under graph isomorphism. The Wiener index (a distance based topological index) is the oldest topological index, which was devised in 1947 for predicting the boiling points of paraffins \cite{Wiener47}. The Platt index ($Pl$), proposed in 1952 for predicting paraffin properties \cite{Platt52}, belongs to the oldest degree based topological indices:
\[Pl(G)=\sum_{uv\in E(G)}(d_{u}+d_{v}-2),\]
where $uv$ is the edge connecting the vertices $u,v$ of the graph $G$, $E(G)$ is the edge set of $G$ and $d_{u}$ is the degree of the vertex $u$. Several authors (see for example \cite{Barysz86,Devillers99}) used the notation $F$ for representing the Platt index. However, this notation is being also used for the forgotten topological index \cite{Furtula15}. Thereby, we use $Pl$, instead of $F$, for representing the Platt index. It is worth mentioning here that the Platt index can be written as $Pl(G)=M_{1}(G)-2m$ where $m$ is the size of $G$ and $M_{1}$ is the first Zagreb index, appeared in 1972 within the study of total $\pi$-electron energy \cite{Gutman72}, which can be defined as:
\[M_{1}(G)=\sum_{uv\in E(G)}(d_{u}+d_{v}).\]
Because of the relation $Pl(G)=M_{1}(G)-2m$ and the following lemma, it is clear that the Platt index is also related to the size of $L(G)$, the line graph of $G$:
\begin{lem}\label{lem1}
\cite{Ha69} If $G$ is an $(n,m)$-graph then the size of $L(G)$ is $\frac{M_{1}(G)}{2}-m$.
\end{lem}
%It is worth mentioning here that $Pl_{2}(G)=EM_{1}(G)$, that is the first reformulated Zagreb index \cite{Milicevic04}.
Gutman and Estrada \cite{Gutman96} proposed line-graph-based topological indices:
if $I(G)$ is a topological index that can be computed from a graph $G$, then there is another
topological index $i(G)$ such that $i(G)=I(L(G))$. Applying this idea to first general Zagreb index
(also known as zeroth-order general Randi\'{c} index) \cite{Li05}, gives general Platt index:
\[Pl_{\alpha}(G)=\sum_{uv\in E(G)}(d_{u}+d_{v}-2)^{\alpha},\]
$\alpha \in \mathbb{R}\setminus \{0\}$. %is a non-zero real number.
It needs to be mentioned here that $Pl_{2}$
coincides with the first reformulated Zagreb index \cite{Milicevic04}.

Zhou and Trinajsti\'{c} \cite{Zhou10} introduced the general sum-connectivity index:
\[\chi_{\alpha}(G)=\sum_{uv\in E(G)}(d_{u}+d_{v})^{\alpha},\]
$\alpha \in \mathbb{R}\setminus \{0\}.$
Note that $\chi_{1}$ is the first Zagreb index and $\chi_{2}$ is the Hyper-Zagreb index \cite{Shirdel13}. Details about general sum-connectivity index can be found in the recent papers \cite{Zhu16,Akhter15,Cui17,Tache15,Tomescu15,Tomescu16} and related references cited therein.

In 2011, Vuki\v{c}evi\'{c} \cite{Vukicevic11} proposed the following topological index (and named it as the variable sum exdeg index) for predicting the octanol-water partition coefficient of certain chemical compounds:
\[SEI_{a}=\sum_{uv\in E(G)}(a^{d_{u}}+a^{d_{v}}),\]
$a \in \mathbb{R^+} \setminus \{ 1 \}$.
%where $a$ is a positive real number different from 1.
Details about variable sum exdeg index can be found in the papers \cite{Ghalavand17,v-ukicevic-11b,Yarahmadi15}.

So far, numerous topological indices have been and are being introduced. A considerable amount of existing degree based topological indices (such as the (general) Platt index, first Zagreb index, variable sum exdeg index, generalized sum-connectivity index, etc.) can be written \cite{hollas,VuDe10} in the following form:
\begin{equation}\label{Eq.1}
\BID(G)=\sum_{uv\in E(G)}\Psi(d_{u},d_{v}).
\end{equation}
where $\Psi$ is a non-negative real valued (symmetric) function of $d_{u}$ and $d_{v}$. The topological indices which have the form (\ref{Eq.1}), are called \textit{bond incident degree indices} \cite{VuDu11} (BID indices in short). The problem of characterizing extremal graphs with respect to certain BID indices among the set of all $n$-vertex tree, unicyclic, bicyclic, tricyclic and tetracyclic graphs is one of the most studied problems in chemical graph theory. In many cases, the extremal graphs for different BID indices are same or have some common properties.
The main purpose of the present note is to attack the aforementioned problem for general BID indices.
Here, it is proved that if a BID index satisfy some certain conditions then the extremal $(n,m)$-graph with respect to the BID index (under consideration) must contain at least one vertex having degree $n-1$. Using this result, graphs having maximum general sum-connectivity index, general Platt index and variable sum exdeg index in the collections of all $n$-vertex tree, unicyclic, bicyclic, tricyclic and tetracyclic graphs are characterized and several existing results \cite{Zhu16,Tache14,Ghalavand17,Gao17,Ji14,Ji15,Ilic12} are obtained as corollaries.

All the graphs considered in the present study are finite, undirected, simple and connected. An $(n,m)$-graph (connected) is said to be a tree, unicyclic graph, bicyclic graph, tricyclic graph, tetracyclic graph if $m=n-1,m=n,m=n+1,m=n+2,m=n+3$ respectively.
%Let $\mathbb{T}_{n}, \mathbb{U}_{n}, \mathbb{B}_{n},\mathbb{TRI}_{n},\mathbb{TET}_{n}$ be the collections of all $n$-vertex trees, unicyclic, bicyclic, tricyclic, tetracyclic graphs respectively.
For a vertex $u\in V(G)$, the set of all vertices adjacent with $u$ is denoted by $N_{G}(u)$ (the neighborhood of $u$). A vertex $v\in V(G)$ of degree 1 is called pendent vertex. A vertex which is not pendent is denoted as non-pendent vertex. As usual, the $n$-vertex star graph is denoted by $S_{n}$. The unique $n$-vertex unicyclic graph obtained from $S_{n}$ by adding an edge is denoted by $S_{n}^{+}$. Undefined notations and terminologies from (chemical) graph theory can be found in \cite{EsBo2013,Ha69,Tr92}.

\section{Main Results}
%\section{Results}

%\begin{lem}\label{lem1}
%\cite{Harary69} If $G$ is an $(n,m)$ graph and $d_{v}$ denotes the degree of $v\in V(G)$ in $G$ then the size of $L(G)$ is $\sum_{v\in V(G)}d_{v}^{2}-m$.
%\end{lem}

In order to obtain the main results, we need to establish some preliminary results. Firstly, we present the following lemma, related to an extension of the function $\Psi$ introduced in (\ref{Eq.1}).

\begin{lem}\label{thm1}
Let $f$ be an extension of the function $\Psi$  to the interval $[0,\infty)$.\\ \\
(i) Let $G$ be the $(n,m)$-graph with maximum BID index and both expressions $f(x+t,y)-f(x,y)+f(c-t,y)-f(c,y)$, $f(x+t,c-t)-f(x,c)$ are non-negative, where $x\geq c> t\geq1$, $c\geq2$ and $y\geq1$. Furthermore, if one of the following two conditions holds:
\begin{itemize}
  \item The function $f$ is increasing in both variables on the interval $[1,\infty)$ and at least one of the expressions $f(x+t,y)-f(x,y)+f(c-t,y)-f(c,y)$, $f(x+t,c-t)-f(x,c)$ is positive.
  \item The function $f$ is strictly increasing in both variables on the interval $[1,\infty)$.
\end{itemize}
Then the maximum vertex degree in $G$ is $n-1$.\\ \\
(ii) Let $G$ be the $(n,m)$-graph with minimum BID index and both expressions $f(x+t,y)-f(x,y)+f(c-t,y)-f(c,y)$, $f(x+t,c-t)-f(x,c)$ are non-positive, where $x\geq c> t\geq1$, $c\geq2$ and $y\geq1$. Furthermore, if one of the following two conditions holds:
\begin{itemize}
  \item The function $f$ is decreasing in both variables on the interval $[1,\infty)$ and at least one of the expressions $f(x+t,y)-f(x,y)+f(c-t,y)-f(c,y)$, $f(x+t,c-t)-f(x,c)$ is negative.
  \item The function $f$ is strictly decreasing in both variables on the interval $[1,\infty)$.
\end{itemize}
Then the maximum vertex degree in $G$ is $n-1$.

%{\color{red}
%(i) Let $G$ be the $(n,m)$-graph with maximum BID index.
%If $f$ is strictly increasing (increasing) in both variables on the interval $[1,\infty)$
%and both (at least one of the) expressions $f(x+t,y)-f(x,y)+f(c-t,y)-f(c,y)$, $f(x+t,c-t)-f(x,c)$ are (is) non-negative
%(positive, respectively) for $x\geq c> t\geq1$, $c\geq2$ and $y\geq1$, then the maximum vertex degree in $G$ is $n-1$.\\ \\
%
%(ii) Let $G$ be the $(n,m)$-graph with maximum BID index.
%If $f$ is strictly decreasing (decreasing) in both variables on the interval $[1,\infty)$
%and both (at least one of the) expressions $f(x+t,y)-f(x,y)+f(c-t,y)-f(c,y)$, $f(x+t,c-t)-f(x,c)$ are (is) non-positive
%(negative, respectively) for $x\geq c> t\geq1$, $c\geq2$ and $y\geq1$, then the maximum vertex degree in $G$ is $n-1$.
%}
\end{lem}
\begin{proof}
For $n \leq 3$, the statement of the theorem is obvious, since it is assumed that $G$ is a connected graph.
Further, we consider   the case $n \geq 4$.
We prove the part ($i$). The part ($ii$) can be proved in fully analogous way.
Suppose to the contrary that every vertex in $G$ has degree at most $n-2$ and let $u$ be a vertex of $G$ with maximum degree.
The assumption $d_{u}\leq n-2$ implies that there exist vertices $v,v_{1}\in V(G)$ such that $uv,vv_{1}\in E(G)$ but $uv_{1}\not\in E(G)$. Let
$N_{G}(v)\setminus N_{G}(u)=\{u,v_{1},v_{2},...,v_{s}\}$ where $s\geq1$.
Suppose that $G^{*}$ is the graph obtained from $G$ by removing the edges
$v_{1}v,v_{2}v,...,v_{s}v$ and adding the edges $v_{1}u,v_{2}u,...,v_{s}u$.
Note that both graphs $G$ and $G^{*}$ have same vertex set.
%In the remaining proof, by the vertex degree $d_{t}$ we mean degree of the vertex $t$ in the graph $G$.
The difference between the values of  BID indices of $G^*$ and $G$ is
%Bearing in mind the definition of the function $f$ and conditions on it, we have
\begin{eqnarray}\label{Eq.2}
\BID(G^{*})-\BID(G)&=&\sum_{ \substack{ w\in N_{G}(u)\setminus N_{G}(v),   \nonumber \\
        w\neq v}}\left(f(d_{u}+s,d_{w})-f(d_{u},d_{w})\right)  \nonumber\\
&&+\sum_{z\in N_{G}(u)\cap N_{G}(v)}\left(f(d_{u}+s,d_{z})-f(d_{u},d_{z})\right)  \nonumber \\
&&+\sum_{z\in N_{G}(u)\cap N_{G}(v)}\left(f(d_{v}-s,d_{z})-f(d_{v},d_{z})\right)  \nonumber \\
&&+\sum_{i=1}^{s}\left(f(d_{u}+s,d_{v_{i}})-f(d_{v},d_{v_{i}})\right)   \nonumber \\
&&+f(d_{u}+s,d_{v}-s)-f(d_{u},d_{v}).  \nonumber
\end{eqnarray}
Bearing in mind the definition of the function $f$ and its constrains, we obtain that the right-hand side of the
above equation is positive, which is a contradiction to the assumption that $G$ is a graph with maximal $\BID$ index.

\end{proof}

In the sequel, we show that several very well established graph topological indices satisfy the constrains on the function $f$,
and thus for them Lemma~\ref{thm1} holds.

First, consider the substitution $\Psi(d_{u},d_{v})=(d_{u}+d_{v})^{\alpha}$ in (\ref{Eq.1}), where $\alpha$ is a non-zero real number.
By this substitution, we obtain the general sum-connectivity index $\chi_{\alpha}$.

\begin{cor}\label{cor1}
For $\alpha \geq1$, if an $(n,m)$-graph $G$ has maximum $\chi_{\alpha}$ value then maximum vertex degree in $G$ is $n-1$.
\end{cor}
\begin{proof}
Let $f(x,y)=(x+y)^{\alpha}$ be the extension of the function $\Psi(d_{u},d_{v})=(d_{u}+d_{v})^{\alpha}$ to the interval $[0,\infty)$, where $\alpha \geq1$. It can be easily verified that the function $f$ is strictly increasing in both $x,y$. It holds that
$f(x+t,c-t)-f(x,c)=0$. If $x\geq c>t\geq1$, $c\geq2$ and $y\geq1$ then, by virtue of Lagrange's Mean Value Theorem, there exist numbers $\Theta_{1},\Theta_{2}$ (where $x+y< \Theta_{1}< x+y+t$ and $y+c-t<\Theta_{2}< y+c$) such that
\[f(x+t,y)-f(x,y)+f(c-t,y)-f(c,y)=\alpha t(\Theta_{1}^{\alpha-1}-\Theta_{2}^{\alpha-1}),\]
which is non-negative because $\Theta_{1}>\Theta_{2}$. The desired result follows from Lemma \ref{thm1}.

%It can be easily verified that the function $f$ is strictly increasing and convex in both $x,y$. Therefore, for $x\geq c>t\geq1$, $c\geq2$ and $y\geq1$,
%it holds that
%\beq \label{eq-sum-con-20}
%f(x+t,y)-f(x,y) >f(c-t,y)-f(c,y).
%\eeq
%The desired result follows since (\ref{eq-sum-con-10}) and (\ref{eq-sum-con-20}) satisfy the rest of the constrains in Theorem \ref{thm1}($i$).

\end{proof}

The choice $\Psi(d_{u},d_{v})=(d_{u}+d_{v}-2)^{\alpha}$ in  (\ref{Eq.1}),  where $\alpha$ is a non-zero real number,
leads to the general Platt index ($Pl_{\alpha}$). The proof of Corollary \ref{cor3},
related to the general Platt index, is fully analogous to that of Corollary \ref{cor1} and will be omitted.

\begin{cor}\label{cor3}
For $\alpha \geq1$, if an $(n,m)$-graph $G$ has maximum $Pl_{\alpha}$ value then maximum vertex degree in $G$ is $n-1$.
\end{cor}

The substitution $\Psi(d_{u},d_{v})=a^{d_{u}}+a^{d_{v}}$ in  (\ref{Eq.1}), where $a$ is a positive real number different from $1$,
leads to the variable sum exdeg index ($SEI_{a}$).

\begin{cor}\label{cor5}\label{cor5}
Let $G$ be an $(n,m)$-graph. If $G$ has maximum $SEI_{a}$ value for $a>1$ then maximum vertex degree in $G$ is $n-1$.
\end{cor}

\begin{proof}
Let $f(x,y)=a^{x}+a^{y}$ be the extension of the function $\Psi(d_{u},d_{v})=a^{d_{u}}+a^{d_{v}}$ to the interval $[0,\infty)$.
It can be easily seen that the function $f$ is strictly increasing in both $x,y$ on the interval $[1,\infty)$ for $a>1$.
If $x\geq c>t\geq1$, $c\geq2$ and $y\geq1$, then
\beq \label{eq-sum-con-40}
f(x+t,y)-f(x,y)+f(c-t,y)-f(c,y)&=&f(x+t,c-t)-f(x,c) \nonumber \\
&=&(a^{t}-1)(a^{x}-a^{c-t}) \nonumber
\eeq
is positive for all $a>1$. The desired result follows from Theorem \ref{thm1}.

\end{proof}

The results, established in this section till now, can be of a significant help by characterization of the extremal graphs
with respect to the general sum-connectivity index, the general Platt index and variable sum exdeg index.
Next, we characterize the graphs with maximal values of the above indices in the case
of tree, unicyclic, bicyclic, tricyclic and tetracyclic graphs.

%Having the results form this section, next we characterize  the graphs with maximal
%general sum-connectivity, general Platt  and variable sum exdeg indices among
%trees, unicyclic, bicyclic, tricyclic and tetracyclic graphs.

\subsection{Extremal tree, unicyclic, bicyclic, tricyclic and tetracyclic graphs}

First we consider the characterization  of the graphs with maximal
general sum-connectivity index $\chi_{\alpha}$.
Recall that $S_{n}^{+}$ denotes the unique $n$-vertex unicyclic graph obtained from $S_{n}$ by adding an edge.

\begin{thm}\label{thm2}
For $\alpha \geq1$, the graphs $S_{n},S_{n}^{+}$ and $B_{1}$ (depicted in Figure \ref{f1})
have the maximum value of the general sum-connectivity index $\chi_{\alpha}$
%among the collection $\mathbb{T}_{n}$ ($n\geq4$), $\mathbb{U}_{n}$ ($n\geq4$),
among the $n$-vertex trees, unicyclic graphs  and bicyclic graphs, respectively.
Among the $n$-vertex tricyclic graphs, either the graph $G_{4}$ or $G_{5}$ (depicted in Figure \ref{f2})
has the maximum $\chi_{\alpha}$ value, for $\alpha\geq1$ and $n \geq 5$.
Namely,
    \begin{equation*}
                \begin{array}{ll}
                    \chi_{\alpha}(G_{4}) < \chi_{\alpha}(G_{5}) & \quad \text{for} \quad n \geq 5, 1 < \alpha < 2, \\
                    \chi_{\alpha}(G_{4}) > \chi_{\alpha}(G_{5}) & \quad \text{for} \quad n \geq 5, \alpha > 2, \\
                    \chi_{\alpha}(G_{4}) = \chi_{\alpha}(G_{5}) & \quad \text{for} \quad  n \geq 5, \alpha =1,2.
                \end{array}
    \end{equation*}
Finally,
%either the graph $H_{4}$ or $H_{5}$ (depicted in Figure \ref{f3}) has the maximum $\chi_{\alpha}$
$H_{4}$, respectively $H_5$, (depicted in Figure \ref{f3}) has the maximum $\chi_{\alpha}$
value, for $\alpha\geq1$ and $n \geq 6$, respectively for $\alpha\geq1$ and $n=5$, among the $n$-vertex tetracyclic graphs.
%The graphs $H_4$ and $H_5$ are depicted in Figure \ref{f3}.
\end{thm}

\begin{figure}
   \centering
    \includegraphics[width=1.9in, height=1.2in]{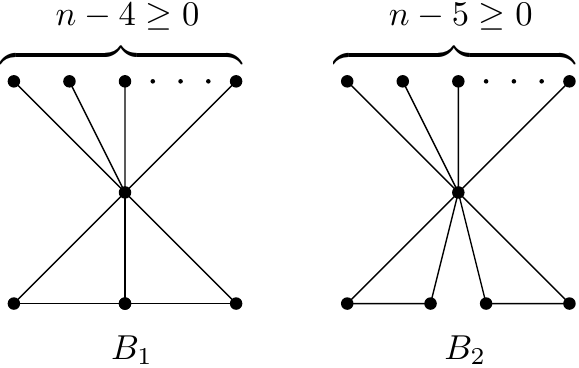}
    \caption{All non-isomorphic $n$-vertex bicyclic graphs with maximum degree $n-1$.}
    \label{f1}
     \end{figure}

\begin{figure}
   \centering
    \includegraphics[width=5in, height=1.3in]{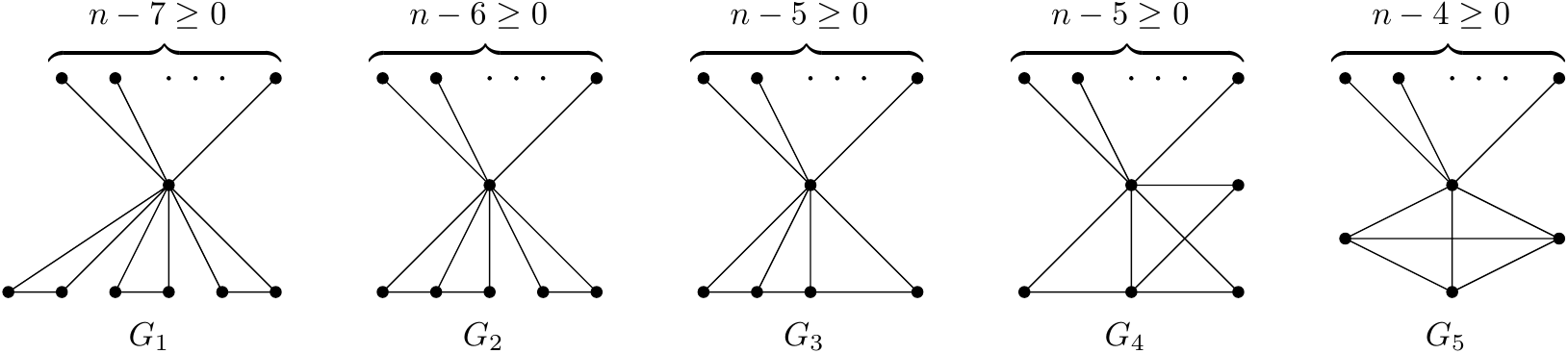}
    \caption{All non-isomorphic $n$-vertex tricyclic graphs with maximum degree $n-1$.}
    \label{f2}
     \end{figure}

\begin{figure}
   \centering
    \includegraphics[width=5in, height=3.9in]{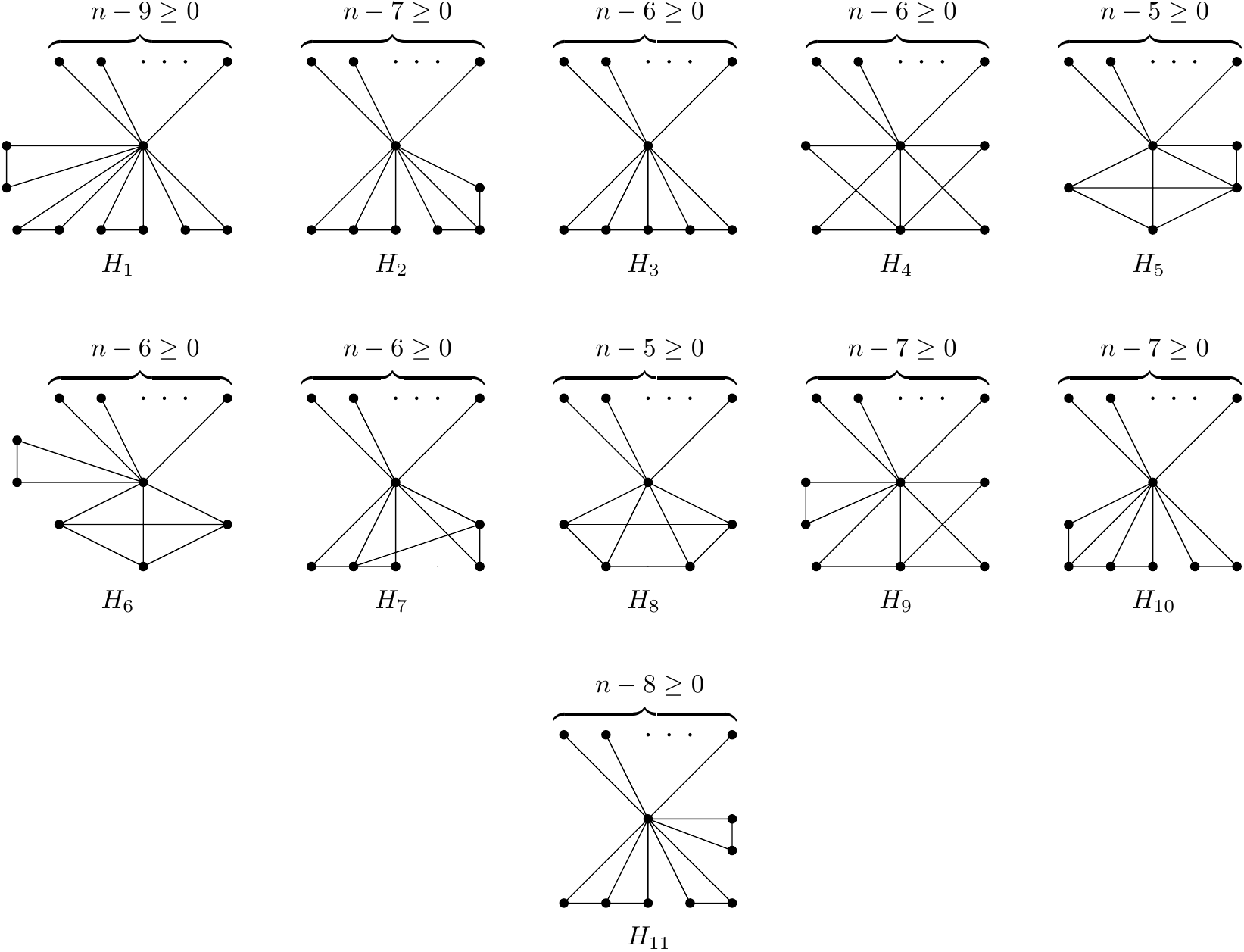}
    \caption{All non-isomorphic $n$-vertex tetracyclic graphs with maximum degree $n-1$.}
    \label{f3}
     \end{figure}

\begin{proof}
Due to Corollary \ref{cor1}, we need to consider only those graphs which have at least one vertex of degree $n-1$.
Notice that there is one $n$-vertex tree and one unicyclic graph with $n$ vertices,
namely $S_{n}$ and $S_{n}^{+}$,
respectively, with maximum vertex degree $n-1$.

If $n =4$, then there is only one bicyclic graph, with maximum degree $n-1=3$, namely $B_1$.
For $n \geq 5$, there are only two non-isomorphic $n$-vertex bicyclic graphs $B_{1}$ and $B_{2}$,
depicted in Figure \ref{f1}, with maximum vertex degree $n-1$.
In the case  $n \geq 5$, we need to compare $\chi_{\alpha}(B_{1})$ and
$\chi_{\alpha}(B_{2})$.
Simple calculations yield:
\[\chi_{\alpha}(B_{1})=2 \cdot 5^{\alpha}+2(n+1)^{\alpha}+(n+2)^{\alpha}+(n-4)n^{\alpha},\]
\[\chi_{\alpha}(B_{2})=2 \cdot 4^{\alpha}+4(n+1)^{\alpha}+(n-5)n^{\alpha}, \quad  \text{and} \]
\beq \label{eq-thm31-10}
\chi_{\alpha}(B_{1})- \chi_{\alpha}(B_{2})=2 (5^{\alpha} - 4^{\alpha})-2(n+1)^{\alpha}+(n+2)^{\alpha}+n^{\alpha}.
\eeq
%
%The first therm, $2 (5^{\alpha} - 4^{\alpha})$ is larger than $0$, for $\alpha\geq1$.
The function $x^\alpha, \alpha \geq 1$, is convex. By Jensen inequality \cite{Jensen-1906}, we obtain
\beqs
&&\frac{1}{2}(n+2)^{\alpha}+\frac{1}{2}n^{\alpha} \geq \left(\frac{1}{2}(n+2)+\frac{1}{2}n \right)^{\alpha} =(n+1)^{\alpha}, \quad  \text{or} \\
&&-2(n+1)^{\alpha}+(n+2)^{\alpha}+n^{\alpha} \geq 0.
\eeqs
The last inequality and the fact that $2 (5^{\alpha} - 4^{\alpha})$ is positive, for $\alpha\geq1$, imply
that the right-hand side of the expression in (\ref{eq-thm31-10}) is positive, and thus, $\chi_{\alpha}(B_{1}) > \chi_{\alpha}(B_{2}) $.
%There exist numbers $\Theta_{0},\Theta_{1}$ satisfying $n<\Theta_{0}<n+1<\Theta_{1}<n+2$ such that
%\[\chi_{\alpha}(B_{1})-\chi_{\alpha}(B_{2})=2(5^{\alpha}-4^{\alpha})+\alpha(\Theta_{1}^{\alpha-1}-\Theta_{0}^{\alpha-1})\]
%which is positive for all $\alpha\geq1$.

All non-isomorphic $n$-vertex tricyclic graphs $G_{1},G_{2},...,G_{5}$
(which can be obtained from $B_{1}$ and $B_{2}$) are shown in Figure~\ref{f2}.
For $n=4$, the graph $G_{5}$ is the only tricyclic graph with maximum degree $n-1=3$.
For $n \geq 5$ %, in a fully analogous manner as in the case of the bicyclic graphs,
the following inequalities can be easily verified for  $\alpha\geq1$:
\[\chi_{\alpha}(G_{4})-\chi_{\alpha}(G_{i})>0,\]
%\[ \chi_{\alpha}(H_{5})-\chi_{\alpha}(H_{8})>0, \ \chi_{\alpha}(H_{4})-\chi_{\alpha}(H_{j})>0,\]
where $i\in\{1,2,3\}$.
Here, we show the above inequality for $i=1$. In fully analogous manner the proof for $i=2,3$ can be obtained.
%Now, the difference $\chi_{\alpha}(G_{4})-\chi_{\alpha}(G_{1})$ for $n\geq7$ will be observed.
By the Lagrange's Mean Value Theorem there must exist numbers $\Theta_{2}, \Theta_{3}$ satisfying $n<\Theta_{2}<n+1<\Theta_{3}<n+3$ such that
\[\chi_{\alpha}(G_{4})-\chi_{\alpha}(G_{1})=3(6^{\alpha}-4^{\alpha})+2\alpha(\Theta_{3}^{\alpha-1}-\Theta_{2}^{\alpha-1}),\]
which is positive for all $\alpha\geq1$.

Next we compare $\chi_{\alpha}(G_{4})$ and  $\chi_{\alpha}(G_{5})$. It holds that
\beq \label{eq-thm31-G-10}
\chi_{\alpha}(G_{4})-\chi_{\alpha}(G_{5}) &=& 3(n+1)^{\alpha}+(n+3)^{\alpha}-n^{\alpha}-3(n+2)^{\alpha}.
\eeq
Applying Newton's generalized binomial theorem, further we  obtain
\beq \label{eq-thm31-G-20}
&& \chi_{\alpha}(G_{4})-\chi_{\alpha}(G_{5}) = \nonumber  \\
 && 3\left( \binom{\alpha}{0} n^{\alpha}+  \binom{\alpha}{1} n^{\alpha-1}
                                                                  +  \binom{\alpha}{2} n^{\alpha-2} +  \binom{\alpha}{3} n^{\alpha-3} +  \binom{\alpha}{4} n^{\alpha-4}+ \cdots \right) \nonumber \\
                                                                   &+&  \binom{\alpha}{0} n^{\alpha}+  \binom{\alpha}{1} n^{\alpha-1} 3
                                                                  +  \binom{\alpha}{2} n^{\alpha-2} 3^2+  \binom{\alpha}{3} n^{\alpha-3}  3^3+   \binom{\alpha}{4} n^{\alpha-4}  3^4 + \cdots  \nonumber \\
                                                                  &-&n^{\alpha} \nonumber \\
                                                                  &-& 3\left( \binom{\alpha}{0} n^{\alpha}+  \binom{\alpha}{1} n^{\alpha-1} 2
                                                                  +  \binom{\alpha}{2} n^{\alpha-2} 2^2+  \binom{\alpha}{3} n^{\alpha-3}  2^3 +  \binom{\alpha}{4} n^{\alpha-4}  2^4+ \cdots \right) \nonumber \\
                                                                  &=& \binom{\alpha}{0} n^{\alpha}(3+1-1-3)+ \binom{\alpha}{1} n^{\alpha-1} (3+3-3 \cdot 2) + \binom{\alpha}{2} n^{\alpha-2}(3+3^2-3\cdot2^2)+\nonumber \\
                                                                  &+&  \binom{\alpha}{3} n^{\alpha-3} (3+3^3-3\cdot2^3) + \binom{\alpha}{4} n^{\alpha-4} (3+3^4-3 \cdot 2^4) +  \cdots \nonumber \\
                                                                  &=&  \sum_{k=0}^{\infty} B_k \nonumber
                                                                  %&=&  6 \binom{\alpha}{3} n^{\alpha-3}  + 36 \binom{\alpha}{4} n^{\alpha-4} +  \cdots = B_3 + B_4+ \cdots
                                                                  %&=& \sum_{k=0}^{\infty} B_k,
\eeq
where $B_0=\binom{\alpha}{0} n^{\alpha} (3+1-1-3)$ and $B_k=\binom{\alpha}{k} n^{\alpha-k} (3+3^k-3\cdot2^k)$, for $k \geq 1$.
Observe that $B_0=B_1=B_2=0$ and thus
\beq \label{eq-thm31-G-30}
 \chi_{\alpha}(G_{4})-\chi_{\alpha}(G_{5}) &=& \sum_{k=3}^{\infty} B_k. \nonumber
\eeq

For $\alpha=1,2$  and $k \geq 3$ the binomial coefficients $\binom{\alpha}{k}=0$, and thus,
$$
\chi_{1}(G_{4})-\chi_{1}(G_{5})=\chi_{2}(G_{4})-\chi_{2}(G_{5})=0.
$$
Since
$
\binom{\alpha}{m} = \frac{\alpha (\alpha -1) \cdots (\alpha -m+1)}{m!} ,
$
it holds that $B_m >0$, if $m < \alpha +1$.
If $m > \alpha +1$, then $B_m$ and $B_{m+1}$ are either both zero or have different sign, according to: $\alpha$ is integer or non-integer, respectively.
Next, we show that $|B_m| >|B_{m+1}|$ for $n \geq 5$, $m \geq 3$ and $m  > \alpha > 1$.
%if $m > \alpha +1$.
%Indeed,
 Expanding $B_m$, we get
\beq \label{eq-thm31-G-40}
|B_m| &=& \left| \binom{\alpha}{m} n^{\alpha-m} (3+3^m-3\cdot2^m) \right| \nonumber \\
&=& \left| \frac{\alpha (\alpha -1) \cdots (\alpha -m+1)}{m!} \right| n^{\alpha-m} (3+3^m-3\cdot2^m).
\eeq
%From $3+3^m-3\cdot2^m >0$, it follows that
It holds that
\beq \label{eq-thm31-G-40-1}
\frac{3}{n}+ \frac{3}{n} \cdot  3^m- \frac{2}{n} \cdot 3 \cdot2^m > 0.
\eeq
%An easy verification shows that for $n \geq 5$ and $m \geq 3$
Next we show that
\beq \label{eq-thm31-G-40-03}
3+3^m-3\cdot2^m  >  \frac{m-1}{m+1} \left(\frac{3}{n}+ \frac{3}{n} \cdot  3^m- \frac{2}{n} \cdot 3 \cdot2^m \right)
\eeq
is satisfied for $n \geq 5$ and $m \geq 3$.
To show that (\ref{eq-thm31-G-40-03}) holds is equivalent as to show that the function
\beq \label{eq-thm31-G-40-03-02}
f(n,m) &= & n(m+1)(3+3^m-3\cdot2^m)  -  (m-1)\left(3+ 3 \cdot  3^m- 2 \cdot 3 \cdot2^m \right)  \nonumber
%          & = & 3(n(m+1)-(m-1))+ 3^m(n(m+1)-3(m-1)) +2^m(6(m-1)-3n(m+1)).
          %     \left(3++3^m-3\ 2^m\right) + 3 \left(-1+2^{1+m}-3^m\right) (-1+m).
\eeq
is positive for $n \geq 5$ and $m \geq 3$.
The function $f(n,m)$ increases in $n$, thus it has its lower bound for $n=5$. After the evaluation
of the function  $f(n,m)$ for $n=5$, we obatin
\beq \label{eq-thm31-G-40-03-03}
f(n,m)  \geq f(5,m) =    18-21\cdot \ 2^m+8\cdot3^m+\left(12-9\cdot2^m+2\cdot3^m\right) m, \nonumber
          %     \left(3++3^m-3\ 2^m\right) + 3 \left(-1+2^{1+m}-3^m\right) (-1+m).
\eeq
which is positive for $m \geq 3$.

From $m  > \alpha > 1$, it follows further that
%From (\ref{eq-thm31-G-40}) and (\ref{eq-thm31-G-40-05}), one can obtain

\beq \label{eq-thm31-G-40-05}
\frac{m-1}{m+1} \left( \frac{3}{n}+ \frac{3}{n} \cdot  3^m- \frac{2}{n} \cdot 3 \cdot2^m\right) > \frac{|\alpha -m|}{m+1}\left(\frac{3}{n}+ \frac{3}{n} \cdot  3^m- \frac{2}{n} \cdot 3 \cdot2^m \right) > 0.
\eeq
From (\ref{eq-thm31-G-40}), (\ref{eq-thm31-G-40-03}) and (\ref{eq-thm31-G-40-05}), we obtain
%Since $n \geq 5$ and $m > \alpha +1$, further it holds that
\beq \label{eq-thm31-G-50}
|B_m|
&>& \left| \frac{\alpha (\alpha -1) \cdots (\alpha -m+1)(\alpha -m)}{m!(m+1)} \right| n^{\alpha-m} \left( \frac{3}{n}+ \frac{3}{n} \cdot  3^m- \frac{2}{n} \cdot 3 \cdot2^m \right)  \nonumber \\
&{\color{red} =}& \left|\binom{\alpha}{m+1} \right| n^{\alpha-m-1} (3+3^{m+1}-3\cdot2^{m+1})  \nonumber \\
&=& |B_{m+1}|. \nonumber
\eeq

For $1 < \alpha < 2$, $B_3$ is negative and the rest components $B_m$, $m \geq 4$, have interchanging signs.
Since it follows that $|B_3| > B_4 > |B_5| > \cdots$, we have that
$ \chi_{\alpha}(G_{4})-\chi_{\alpha}(G_{5}) < 0$, or  $\chi_{\alpha}(G_{4}) <\chi_{\alpha}(G_{5})$.

If $\alpha$ is a non-integer larger than $2$, then all components $B_m$, $m <  \alpha+1$ are positive.
The first negative component is $B_s$,  $s =  \lceil \alpha+1 \rceil$.
It holds that all components $B_{q}, q > s$ have further interchanging signs.
Also it holds that $B_{s-1} > |B_{s}|  > B_{s+1}  > |B_{s+2}| > \cdots$, and therefore,
$ \chi_{\alpha}(G_{4})-\chi_{\alpha}(G_{5}) > 0$, or  $\chi_{\alpha}(G_{4}) >\chi_{\alpha}(G_{5})$.

%As it mentioned above, for $ \alpha = 1, 2$  all components equals to $0$.
If $ \alpha$ is an integer larger than $2$, then all components $B_m$, $m <  \alpha+1$ are positive, and
all components $B_s$, $s \geq \alpha+1$, are $0$.
Thus, here also $\chi_{\alpha}(G_{4}) >\chi_{\alpha}(G_{5})$.
%

%For $n \geq 5$, it can be verified that
%$\chi_{\alpha}(G_{5})=\chi_{\alpha}(G_{4})$, for $\alpha =1, 2$;
%$\chi_{\alpha}(G_{5})>\chi_{\alpha}(G_{4})$, for $1 < \alpha < 2 $; and
%$\chi_{\alpha}(G_{5})<\chi_{\alpha}(G_{4})$ for $\alpha > 2$.

All non-isomorphic
$n$-vertex tetracyclic graphs $H_{1},H_{2},...,H_{11}$ (which can be obtained from
$G_{1},G_{2},...,G_{5}$), with  maximum vertex degree $n-1$, are shown in Figure~\ref{f3}.
In an analogous manner as in the case of tricyclic graphs, the following inequalities can be verified for  $\alpha\geq1$ and $n \geq 6$:
%\[ \chi_{\alpha}(H_{5})-\chi_{\alpha}(H_{8})>0, \ \chi_{\alpha}(H_{4})-\chi_{\alpha}(H_{j})>0, \]
\[  \chi_{\alpha}(H_{5})-\chi_{\alpha}(H_{8})>0, \ \ \ \chi_{\alpha}(H_{4})-\chi_{\alpha}(H_{j})>0,\]
%where $j\in\{1,2,3,...,11\}\setminus\{5,8\}$. This completes the proof.
where $j\in\{1,2,3,...,11\}\setminus\{4,5,8\}$.

Now, we prove that the difference $\chi_{\alpha}(H_{4})-\chi_{\alpha}(H_{5})$ is positive for $\alpha\geq1$ and $n \geq 6$.
Routine calculations yield
\begin{equation}\label{Eq.123}
\chi_{\alpha}(H_{4})-\chi_{\alpha}(H_{5})=2(7^{\alpha}-6^{\alpha})+(n+4)^{\alpha}-2(n+2)^{\alpha}
+3(n+1)^{\alpha}-(n+3)^{\alpha}-n^{\alpha}.
\end{equation}
For $\alpha=1,2$, the right hand side of Equation (\ref{Eq.123}) is positive,
 and thus, the claim of the theorem in this case holds.

Next we prove the claim of the theorem for tetracyclic graphs
in the cases  $1 < \alpha < 2$  and $\alpha>2$, $n \geq 6$.
Applying Newton's generalized binomial theorem in Equation (\ref{Eq.123}), we have
\begin{equation}\label{Eq.124}
\chi_{\alpha}(H_{4})-\chi_{\alpha}(H_{5})=2(7^{\alpha}-6^{\alpha})+\sum_{k=2}^{\infty} A_k,
\end{equation}
where $A_k=\binom{\alpha}{k} n^{\alpha-k}\left(4^k-2^{k+1}+3-3^k\right)$.

Now, we prove that the inequality $|A_m| >|A_{m+1}|$ holds for $m\ge2, n\ge6$ and $\alpha>1$.
Firstly, we show that
\begin{equation}\label{EQ-AA0}
4^{m} - 2^{m+1} + 3 - 3^{m} > \frac{m-1}{m+1}\cdot\frac{4^{m+1} - 2^{m+2} + 3 - 3^{m+1}}{n},
\end{equation}
for $n\ge6$ and $m\ge2$.
The Inequality (\ref{EQ-AA0}) holds if and only if the function
\beq \label{eq-thm31-G-40-03-02A}
g(n,m) &= & n(m+1)(4^{m} - 2^{m+1} + 3 - 3^{m} )  -  (m-1)\left(4^{m+1} - 2^{m+2} + 3 - 3^{m+1} \right)  \nonumber
\eeq
is positive valued for $n \geq 6$ and $m \geq 2$.
The function $g$ is increasing in $n$, thus it attains minimum value at $n=6$ and thereby we have
\beq \label{eq-thm31-G-40-03-03A}
g(n,m)  \geq 21-8\cdot \ 2^{m+1}-9\cdot3^m + 10\cdot 4^m +\left(15-4\cdot2^{m+1}-3\cdot3^m + 2\cdot4^m\right) m, \nonumber
\eeq
which is positive for $m \geq 2$. Bearing in mind the assumption $m  > \alpha > 1$ and the fact that the inequality $4^{m+1} - 2^{m+2} + 3 - 3^{m+1} > 0$ holds for $m\ge2$, we have
\beq \label{EQ-AA2}
\frac{m-1}{m+1} \left(\frac{4^{m+1} - 2^{m+2} + 3 - 3^{m+1}}{n}\right) > \frac{|\alpha -m|}{m+1}\left(\frac{4^{m+1} - 2^{m+2} + 3 - 3^{m+1}}{n} \right).
\eeq
Using (\ref{EQ-AA0}) and (\ref{EQ-AA2}), we obtain
\beq \label{eq-thm31-G-40}
|A_m| &=& \left| \frac{\alpha (\alpha -1) \cdots (\alpha -m+1)}{m!} \right| n^{\alpha-m} (4^{m} - 2^{m+1} + 3 - 3^{m}) \nonumber \\
&>& |A_{m+1}|.
\eeq
For $1 < \alpha < 2$, $A_2$ is positive and the remaining components $A_3, A_4, A_5, \cdots$ have interchanging signs.
Since we have just proved that $A_2 > |A_3| > A_4 > |A_5| > \cdots$, we have
$ \chi_{\alpha}(H_{4})-\chi_{\alpha}(H_{5}) > 0$.

If $\alpha > 2$ then all components $A_p$ (for $p <  \alpha+1$) are positive. If $\alpha$ is an integer larger than $2$, then all components $A_q$ (for $q \geq \alpha+1$) are $0$. If $\alpha$ is a non-integer larger than $2$, then the first negative component is $A_r$, where  $r =  \lceil \alpha+1 \rceil$.
All the components $A_{r}, A_{r+1}, A_{r+2}, \cdots$ have interchanging signs. Also, it holds that $A_{r-1} > |A_{r}|  > A_{r+1}  > |A_{r+2}| > \cdots$, and therefore we have $\chi_{\alpha}(H_{4})-\chi_{\alpha}(H_{5}) > 0$ for $\alpha > 2$.

%{\color{red} Next, we consider the case $1 \leq \alpha  \leq 2$.}

Finally, we prove the result for tetracyclic graphs  when $n=5$.
There are only two 5-vertex tetracyclic graphs, namely $H_5$ and $H_8$.
 %having maximum degree 4 and the inequality $\chi_{\alpha}(H_{5})-\chi_{\alpha}(H_{8})>0$ holds for $\alpha\geq1$and $n=5$.
 In this case $\chi_{\alpha}(H_{5})-\chi_{\alpha}(H_{8}) = 8^{\alpha} - 6^{\alpha}$, which is positive for $\alpha\geq1$.
 Hence, the proof of the theorem is completed.
%Hence, we have the desired result.

\end{proof}

\begin{rem}\label{rem1}
To the best of our knowledge, there are no results in literature concerning maximal
general sum-connectivity index, for $\alpha \geq1$, of  unicyclic and tetracyclic graphs.
%Theorem~\ref{thm2} gives a contribution in this direction.
Partially Theorem~\ref{thm2} has already been proved:
see \cite{Zhou10} for trees, \cite{Tache14} for bicyclic graphs and  \cite{Ali17,Zhu16} for tricyclic graphs.
However, the proof of Theorem~\ref{thm2} is significantly shorter than the proof of the main result
established in \cite{Tache14}.
Also, for $\alpha \geq1$, the comparison between
$\chi_{\alpha}(G_{4})$ and $\chi_{\alpha}(G_{5})$ was left as an open problem in both the references \cite{Ali17,Zhu16}. Here, this problem has been solved.

%this comparison has been completed.

\end{rem}

%\begin{rem}\label{rem1a}
%To the best of our knowledge, there is no result concerning maximum general sum-connectivity index (with $\alpha \geq1$) of $n$-vertex unicyclic and tetracyclic graphs. Corollary \ref{cor2} gives a contribution in this direction. Tache \cite{Tache14} proved that ??? has the the the maximum general sum-connectivity index ($\alpha \geq1$) among the collections $\mathbb{B}_{n}$, by using a set of graph transformations. Also, Zhu and Lu \cite{Zhu16} found first fourth largest value of the general sum-connectivity index ($\alpha \geq1$) using transformations. Here, we have proved (in Corollary \ref{cor2}) the these results by giving a considerable short proof.
%\end{rem}

\begin{rem}\label{rem1b}
Recall that $\chi_{2}$ is same as the Hyper-Zagreb index \cite{Shirdel13}.
Recently, Gao \textit{et al.} \cite{Gao17} determined the graphs having maximum $\chi_{2}$
%value among the collections $\mathbb{T}_{n}$, $\mathbb{U}_{n}$ and $\mathbb{B}_{n}$.
value among trees, unicyclic and bicyclic graphs.
Theorem~\ref{thm2} provides a short and unified approach to all these results.
\end{rem}

Next, we  consider
%the characterization  of
the graphs with maximal
general Platt index $Pl_{\alpha}$.
The proof of next result is fully analogous to that of Theorem \ref{thm2} and will be omitted.
The extremal graphs were already mentioned previously and are depicted in
Figures~\ref{f1}, \ref{f2} and  \ref{f3}.

\begin{thm}\label{thm4}
For $\alpha \geq1$, the graphs $S_{n},S_{n}^{+}$ and $B_{1}$
have the maximum value of the general Platt index $Pl_{\alpha}$
%among the collection $\mathbb{T}_{n}$ ($n\geq4$), $\mathbb{U}_{n}$ ($n\geq4$),
among the $n$-vertex trees, unicyclic graphs  and bicyclic grpahs, respectively.
Among the $n$-vertex tricyclic graphs, either the graph $G_{4}$ or $G_{5}$
has the maximum $Pl_{\alpha}$ value, for $\alpha\geq1$ and $n \geq 5$.
Particularly,
    \begin{equation*}
                \begin{array}{ll}
                    Pl_{\alpha}(G_{4}) < Pl_{\alpha}(G_{5}) & \quad \text{for} \quad n \geq 5, 1 < \alpha < 2, \\
                    Pl_{\alpha}(G_{4}) > Pl_{\alpha}(G_{5}) & \quad \text{for} \quad n \geq 5, \alpha > 2, \\
                    Pl_{\alpha}(G_{4}) = Pl_{\alpha}(G_{5}) & \quad \text{for} \quad  n \geq 5, \alpha =1,2.
                \end{array}
    \end{equation*}
%Finally,
%either the graph $H_{4}$ or $H_{5}$ (depicted in Figure \ref{f3}) has the maximum $\chi_{\alpha}$
%$H_{4}$ (respectively $H_5$) has the maximum $Pl_{\alpha}$
%value, for $\alpha\geq1$ and $n \geq 6$ (respectively, for $\alpha\geq1$ and $n=5$), among the $n$-vertex tetracyclic graphs.
Finally,
$H_{4}$, respectively $H_5$, (depicted in Figure \ref{f3}) has the maximum $\chi_{\alpha}$
value, for $\alpha\geq1$ and $n \geq 6$, respectively for $\alpha\geq1$ and $n=5$, among the $n$-vertex tetracyclic graphs.
\end{thm}

%\begin{proof}
%All the non-isomorphic $n$-tricyclic graphs are depicted in Fig. ???. Simple calculations yield:
%\[Pl_{\alpha}(G_{1})=3\times2^{\alpha}+6(n-1)^{\alpha}+(n-7)(n-2)^{\alpha} \text{ \ where \ } n\geq7,\]
%\[Pl_{\alpha}(G_{2})=2^{\alpha}+2\times3^{\alpha}+4(n-1)^{\alpha}+n^{\alpha}+(n-6)(n-2)^{\alpha} \text{ \ where \ } n\geq6,\]
%\[Pl_{\alpha}(G_{3})=4^{\alpha}+2\times3^{\alpha}+2(n-1)^{\alpha}+2n^{\alpha}+(n-5)(n-2)^{\alpha} \text{ \ where \ } n\geq5,\]
%\[Pl_{\alpha}(G_{4})=3\times4^{\alpha}+3(n-1)^{\alpha}+(n+1)^{\alpha}+(n-5)(n-2)^{\alpha} \text{ \ where \ } n\geq5,\]
%\[Pl_{\alpha}(G_{5})=3\times4^{\alpha}+3n^{\alpha}+(n-4)(n-2)^{\alpha} \text{ \ where \ } n\geq4.\]

%\end{proof}

\begin{rem}\label{rem2}
Recall that $Pl_{1}(G)=Pl(G)=M_{1}(G)-2m$. Ili\'{c} and Zhou \cite{Ilic12} proved that among the line graphs of all $n$-vertex unicyclic graphs, the line graph of $S_{n}^{+}$ has the maximum number of edges. This result also follows from Lemma \ref{lem1} and from the fact that $S_{n}^{+}$ is the unique graph with maximum $Pl$ (and $M_{1}$) value among unicyclic graphs.
%the collection $\mathbb{U}_{n}$.
\end{rem}

\begin{rem}\label{rem2a}
Ili\'{c} and Zhou \cite{Ilic12} proved that $S_{n}^{+}$ is the unique graph with maximum first reformulated Zagreb index
(which coincides with $Pl_{2}$) among unicyclic graphs.
%the collection $\mathbb{U}_{n}$.
Ji \textit{et al.} \cite{Ji14} (and Ji \textit{et al.} \cite{Ji15}) characterized the graph
having maximum $Pl_{2}$ value among trees, unicyclic graphs, bicyclic graphs (and among tricyclic graphs respectively).
%the collections  $\mathbb{T}_{n}$, $\mathbb{U}_{n}$, $\mathbb{B}_{n}$ (and among the collection $\mathbb{TRI}_{n}$ respectively).
Theorem \ref{thm4} provides a short and unified approach to all these aforementioned results of \cite{Ji14,Ji15}.
\end{rem}

Finally, we  characterize  trees, unicyclic, bicyclic, tricyclic and tetracyclic graphs  with maximal
variable sum exdeg index $SEI_{a}, a > 1$.

\begin{thm}\label{thm6}
For $a>1$, the graph $S_{n},S_{n}^{+},B_{1},G_{4},H_{4}$ has the maximum $SEI_{a}$ value among trees, unicyclic graphs, bicyclic, tricyclic and tetracyclic graphs respectively,
%the collection $\mathbb{T}_{n}, \mathbb{U}_{n}, \mathbb{B}_{n},\mathbb{TRI}_{n},\mathbb{TET}_{n}$,
respectively.
\end{thm}

\begin{proof}
Due to Corollary \ref{cor5}, we need to consider only those graphs which have at least one vertex of degree $n-1$.
Notice that there is only one $n$-vertex tree and only one unicyclic graph, namely $S_{n}$ and $S_{n}^{+}$ respectively,
having maximum vertex degree $n-1$. Also, there are only two non-isomorphic $n$-vertex bicyclic graphs $B_{1}$ and $B_{2}$,
depicted in Figure \ref{f1}. Moreover, all non-isomorphic $n$-vertex tricyclic graphs $G_{1},G_{2},...,G_{5}$ and
all non-isomorphic $n$-vertex tetracyclic graphs $H_{1},H_{2},...,H_{11}$ are shown in Figure \ref{f2} and Figure \ref{f3}, respectively.
Direct calculations yield:
$SEI_{a}(B_{1})-SEI_{a}(B_{2})=a(1-4a+3a^{2}), SEI_{a}(G_{4})-SEI_{a}(G_{1})=2a(1-3a+2a^{3}),
SEI_{a}(G_{4})-SEI_{a}(G_{2})=a(1-2a-3a^{2}+4a^{3}), SEI_{a}(G_{4})-SEI_{a}(G_{3})=2a^{2}(1-3a+2a^{2}),
SEI_{a}(G_{4})-SEI_{a}(G_{5})=a(-1+6a-9a^{2}+4a^{3})$, which are all positive for $a>1$. %$SEI_{a}(H_{4})-SEI_{a}(H_{1})=2a^{2}(1-3a+2a^{2}),$
After similar calculations, we have $SEI_{a}(H_{4})-SEI_{a}(H_{i})>0$ for all $i\in\{1,2,3,...,11\}$ and $a>1$, which completes the proof.
\end{proof}

\begin{rem}\label{rem3}
Recently, Ghalavand and Ashrafi \cite{Ghalavand17} determined the graphs with maximum $SEI_{a}$ value for $a>1$ among all the $n$-vertex trees, unicyclic, bicyclic and tricyclic graphs.
They claimed that the graph $G_{3}$ has the maximum $SEI_{a}$ value for $a>1$
among tricyclic graphs.
%the collection $\mathbb{TRI}_{n}$.
However, Theorem~\ref{thm6} contradicts this claim.
Also, Theorem~\ref{thm6} provides alternative proofs of four main results proved in \cite{Ghalavand17}.
Moreover, here, we have extended the results from \cite{Ghalavand17} for tetracyclic graphs.
\end{rem}

\section{Concluding remarks}

One of the most studied problems in chemical graph theory is the problem of characterizing extremal elements,
with respect to some certain topological indices, in the collection of all $n$-vertex trees, unicyclic, bicyclic, tricyclic and tetracyclic
graphs. In the present study, we have addressed this problem for the BID indices
(which is a subclass of the class of all degree based topological indices).
We have been able to characterize the graphs with the maximum general sum-connectivity index ($\chi_\alpha$) for $\alpha\geq1$,
general Platt index ($Pl_\alpha$) for $\alpha\geq1$ and variable sum exdeg index ($SEI_{a}$) for $a>1$
in the aforementioned graph collections. Some of these results are new while the already existing
results have been proven in a shorter and unified way.
%The main result of this note is
Lemma~\ref{thm1} confirms the existence of a vertex of degree $n-1$ in the extremal graph, with respect to some
BID indices that satisfy certain conditions. % among all $(n,m)$-graphs.

As in most of the cases, the extremal (either maximum or minimum) $(n,m)$-graph with respect
to a topological index has a vertex of degree $n-1$. Thereby, it would be interesting
to prove this fact for the BID indices other than those considered here
% (other than $\chi_\alpha$, $Pl_\alpha$ and $SEI_{a}$ where $\alpha\geq1$, $a>1$)
as well as for other type of topological indices
(such as spectral indices, eccentricity-based indices, distance-based indices, degree-based indices etc.).
%For example,
%{\color{green} an open problem is}
%one may try
%to establish such a result for the atom-bond connectivity (ABC) index \cite{Ashrafi14,Ashrafi16}
%because  {\color{green} a} $(n,m)$-graph (for $m=n-1,n,...,n+3$) having maximum ABC index contains a vertex of degree $n-1$.

%Then the problem of characterizing extremal elements, with respect to the considered topological index, in the collection of all $n$-vertex trees, unicyclic, bicyclic, tricyclic and tetracyclic graphs may be solved by the technique adopted in the present study.

\end{document}